\documentclass[a4paper,12pt]{article}
\usepackage{amsmath,amssymb,amsthm,amsfonts,mathrsfs}
\usepackage{latexsym,bm,graphicx,color}
\usepackage{float}
\usepackage{geometry}

\newtheorem{thm}{Theorem}[section]
\newtheorem{lem}[thm]{Lemma}

\title{Some novel  minimax results for perfect matchings of hexagonal systems
\thanks{This work is supported by NSFC (grant no. 11871256).}}
\author{Xiangqian Zhou, \quad  Heping Zhang\footnote{Corresponding author.}\\[6pt]
\small{School of Mathematics and Statistics, Lanzhou University,
Lanzhou, Gansu 730000, P. R. China}\\[6pt]
\small{E-mail addresses: zhouxq12@lzu.edu.cn,  zhanghp@lzu.edu.cn}}
\date{}

\begin{document}

\maketitle


\begin{abstract}
The anti-forcing number of a perfect matching $M$ of a graph $G$ is
the minimum number of edges of $G$ whose deletion results in a subgraph
with a unique perfect matching $M$, denoted by $af(G,M)$.
When $G$ is a plane bipartite graph, Lei et al. established a minimax result:
For any perfect matching $M$ of $G$,
$af(G,M)$ equals the maximum number of $M$-alternating cycles of $G$
where any two either are disjoint or intersect only at edges in $M$;
For a hexagonal system, the maximum anti-forcing number equals the fries number. In this paper we show that for every perfect matching $M$ of a hexagonal system $H$  with the maximum anti-forcing number or minus one, $af(H,M)$ equals the  number of $M$-alternating hexagons of $H$. Further we show that  a hexagonal system $H$ has a triphenylene as nice subgraph if and only  $af(H,M)$ always equals the  number of $M$-alternating hexagons of $H$  for every perfect matching $M$ of $H$.

\vspace{0.3cm} \noindent\textbf{Keywords:}
Hexagonal system; Perfect matching; Anti-forcing number; Fries number; Triphenylene.

\end{abstract}


\section{Introduction}

All graphs considered in this paper are finite and simple connected graphs.
Let $G$ be a graph with vertex set $V(G)$ and edge set $E(G)$.
A \emph{perfect matching}  $M$ of $G$ is a set of edges of $G$
such that each vertex of $G$ is incident with exactly one edge in $M$.
This graph-theoretical concept coincides with a Kekul\'{e} structure of chemical molecules.

In 1987, Randi\'{c} and Klein \cite{Klein-Randic-1987} proposed
the \emph{innate degree of freedom} of a Kekul\'{e} structure,
i.e. the minimum number of double bonds which simultaneously belong to
the given Kekul\'{e} structure and to no other one.
This notion has arisen in the study of finding resonance structures of a given molecule in chemistry.
Later,  it is named ``forcing number'' by Harary et al. \cite{Harary-1991}.

A \emph{forcing set} $S$ for a perfect matching $M$ of a graph $G$ is a subset of $M$ which
is not contained in other perfect matchings of $G$.
The \emph{forcing number} of $M$ is the smallest cardinality over all forcing sets of $M$,
denoted by $f(G,M)$.
The \emph{maximum forcing number} of $G$ is the maximum value of forcing numbers
of all perfect matchings of $G$, denoted by $F(G)$.
For further information on this topic,
we refer the reader to a survey \cite{Che-Chen-2011} and other references
\cite{Adams-2004,Afshani-2004,Diwan2019,Pachter-1998,
Riddle-2002,Zhang-Li-1995, WYZ16}.

Let $M$ be a perfect matching of a graph $G$.
A cycle (resp.  a path $P$) $C$ of $G$ is \emph{$M$-alternating} if
the edges of $C$ (resp. $P$) appear alternately in $M$ and $E(G)\backslash M$.
It was revealed \cite{Adams-2004,Riddle-2002} that
a subset $S\subseteq M$ is a forcing set of $M$ if and only if
$S$ contains at least one edge of each $M$-alternating cycle of $G$.
This implies a simple inequality $f(G,M)\geq c(G,M)$,
where $c(G,M)$ denotes the maximum number of disjoint $M$-alternating cycles in $G$.
In the case where $G$ is a plane bipartite graph, Pachter and Kim \cite{Pachter-1998} observed that
these two numbers are always equal to each other.

\begin{thm}\label{Forcing-feedback}
\cite{Pachter-1998}
Let $G$ be a plane bipartite graph with a perfect matching $M$. Then $$f(G,M) = c(G,M).$$
\end{thm}

A more general result on bipartite graphs due to B. Guenin and R. Thomas is given as follows; see  \cite[Corollary 5.8]{Guenin-2011}.

\begin{thm}\cite{Guenin-2011}
 Let $G$ be a bipartite graph, and let $M$ be a perfect matching in $G$. Then
$G$ has no matching minor isomorphic to $K_{3,3}$ or the Heawood graph if and only if
$f(G',M') = c(G',M')$ for every subgraph $G'$ of $G$ such that $M' =
M\cap E(G')$ is a perfect matching in $G'$.
\end{thm}

In 2007, the anti-forcing number of a graph
was  introduced by Vuki\v{c}evi\'{c} and Trinajsti\'{c} \cite{Vukicevic-2007} as
  the smallest number of edges whose
removal results in a subgraph with a unique perfect matching. In an early paper, Li
\cite{Li-1997} proposed a forcing single edge
(i.e. anti-forcing edge) of a graph, which belongs to all but one perfect matching.
For other researches on this topic, see Refs
\cite{Che-Chen-2011,Deng-2007,Vukicevic-2008,Bian-2011}.

More recently, by an analogous manner as the forcing number, Klein and Rosenfeld \cite{Klein-Rosenfeld-2014} and
Lei et al. \cite{Lei-Zhang} independently defined the anti-forcing number of a single perfect matching in a graph.
Let $M$ be a perfect matching of a graph $G$. A subset $S\subseteq E(G)\backslash M$ is
called an \emph{anti-forcing set} of $M$ if $G-S$ has a unique perfect matching $M$,
where $G-S$ denotes the graph obtained from $G$ by deleting all edges in $S$.
The following lemma shows an equivalent condition.

\begin{lem}\label{Anti-forcing-cycle}
\cite{Lei-Zhang}
Let $G$ be a graph and  $M$ be a perfect matching of $G$.
A subset $S$ of $E(G)\backslash M$ is an anti-forcing set of $M$ if and only if
$S$ contains at least one edge of every $M$-alternating cycle.
\end{lem}

The minimum cardinality of anti-forcing sets of $M$ is called the \emph{anti-forcing number} of $M$, denoted by $af(G,M)$. From these concepts, we can see that the anti-forcing number
of a graph $G$ is just the minimum value of anti-forcing numbers of all perfect matchings of $G$.
The \emph{maximum anti-forcing number} of $G$ is  the maximum value of anti-forcing numbers
of all perfect matchings of $G$, denoted by $Af(G)$. Two sharp upper bounds on maximum anti-forcing number and anti-forcing spectrum of a graph, we may refer to recent refs. \cite{DengZhang17b, shiZ2017, DK2015}.

Given a graph $G$ with a perfect matching $M$, two $M$-alternating cycles of $G$
are said to be \emph{compatible} if they either are disjoint or intersect only at edges in $M$.
A set $\mathcal{A}$ of pairwise compatible $M$-alternating cycles of $G$ is called
a \emph{compatible $M$-alternating set}.
Let $c'(G,M)$ denote the maximum cardinality of compatible $M$-alternating sets of $G$. We also have $af(G,M)\geq c'(G,M)$.
For plane bipartite graphs $G$, Lei et al. \cite{Lei-Zhang} established the equality.

\begin{thm}\label{Anti-forcing-feedback}
\cite{Lei-Zhang}
Let $G$ be a plane bipartite graph with a perfect matching of $M$. Then $$af(G,M) = c'(G,M).$$
\end{thm}

A \emph{hexagonal system} (or benzenoid) is a 2-connected finite plane graph such that
every interior face is bounded by a regular hexagon of side length one \cite{Sachs-1984}.
It can also be formed by a cycle with its interior in the infinite hexagonal lattice
on the plane (graphene) \cite{Cyvin-book-1988}.
A hexagonal system with a perfect matching is regarded as
a molecular graph (carbon-skeleton) of a benzenoid hydrocarbon.
Hence these kinds of graphs are called
benzenoid systems and have been extensively investigated;
We may refer to a detailed review \cite{Randic-review-2003} due to Randi\'{c}.

Recently it was known \cite{Xu-2013,Lei-Zhang} that maximum forcing number and anti-forcing number of a hexagonal system are equal to the famous Clar number and Fries number respectively, which can measure the stability of polycyclic benzenoid hydrocarbons
\cite{Abeledo-2007,Clar-1972, Hansen-Zheng-1992}. The same results hold for (4,6)-fullerene graphs \cite{shiWZ2017}.

Let $H$ be a hexagonal system with a perfect matching. A subgraph $H_0$ of $H$ is called {\em nice} if $H-V(H_0)$ has a perfect matching.
A set of disjoint hexagons of $H$ is called {\em sextet pattern} if they form a nice subgraph of $H$.
The size of a maximum resonant set of $H$  is the \emph{Clar number} of $H$, denoted by $Cl(H)$.


\begin{thm}\label{HS-F(H)-Cl(H)}
\cite{Xu-2013}
 Let $H$ be a hexagonal system with a perfect matching. Then $F(H)=Cl(H)$.
\end{thm}

Xu et al. \cite{Xu-2013} obtained the theorem by using   Zheng and Chen's result \cite{Zheng-1985} that if  $H- K$
has at least two different perfect matchings for a resonant set $K$ of $H$, then $Cl(H)\geq |K|+1$.
By rising  this bound by one, the present authors obtained a stronger result.

\begin{thm}\label{HS-F(H)-Cl(H)-stronger}
\cite{Zhou-Zhang-2015}
Let $H$ be a hexagonal system with a perfect matching.
For every perfect matching $M$ of $H$ with forcing number $F(H)$,
$H$ has $F(H)$ disjoint $M$-alternating hexagons.
\end{thm}

However,  Theorem  \ref{HS-F(H)-Cl(H)-stronger} does not necessarily hold  for perfect matchings of $H$
with the  forcing number $F(H)-1$.
For example,  the Coronene (see  Figure \ref{fig:counterexample-1}(a)) has the maximum forcing number 3.
For the specific perfect matching of Coronene marked by the bold lines, it
has forcing number two, but it has only one alternating hexagon.

\begin{figure}[H]
\begin{center}
\includegraphics[height=3cm]{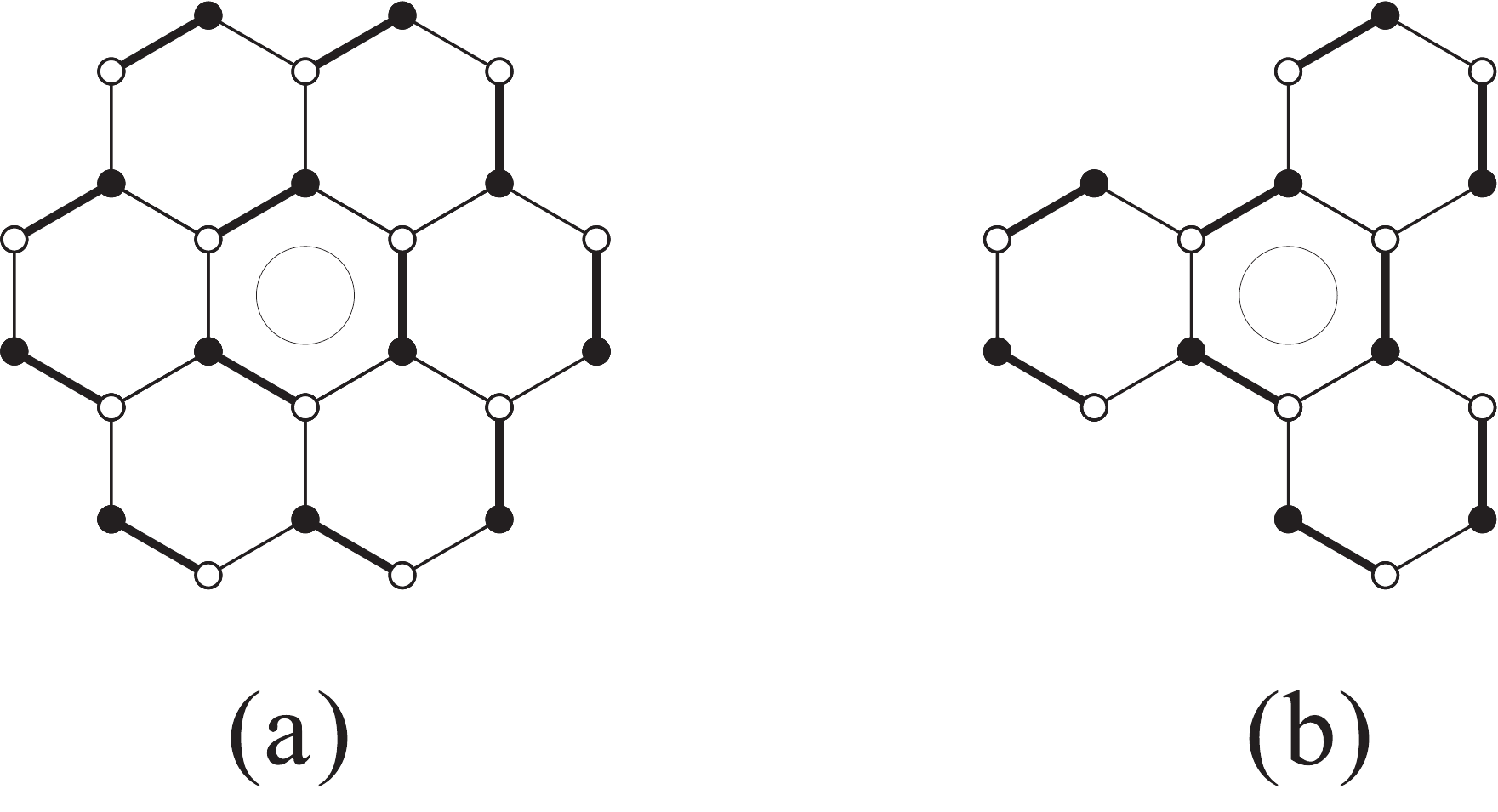}
\caption{{\small (a) Coronene}, and (b) Triphenylene.}
\label{fig:counterexample-1}
\end{center}
\end{figure}

For a hexagonal system $H$ with a perfect matching $M$, let $fr(H,M)$ denote the number of
all $M$-alternating hexagons in $H$.  The maximum value of $fr(H,M)$
over all perfect matchings $M$ is just the {\em Fries number} of $H$, denoted by $Fr(H)$.
Since all $M$-alternating hexagons of $H$ are compatible,
Theorem \ref{Anti-forcing-feedback} implies $af(H,M)=c'(H,M)\geq fr(H,M)$.
The second equality does not  hold in general.
For example, the bold lines of Triphenylene in Figure \ref{fig:counterexample-1}(b)
constitute a perfect matching with anti-forcing number two,
whereas it has only one alternating hexagon.
However, Lei et al. \cite{Lei-Zhang} obtained the following result by finding a perfect matching $M$
with the equality.

\begin{thm}\label{HS-F(H)-Fries(H)}
\cite{Lei-Zhang}
Let $H$ be a hexagonal system with a perfect matching. Then $$Af(H)=Fr(H).$$
\end{thm}

In this article, we show that $af(H,M) = fr(H,M)$ always holds for every perfect matching $M$
of a hexagonal system $H$ with the  maximum anti-forcing number or minus one.

\begin{thm}\label{main-thm}
Let $H$ be a hexagonal system with a perfect matching. Then for every perfect matching $M$ of $H$
with anti-forcing number $Af(H)$ or $Af(H)-1$, we have
 \begin{equation}af(H,M)=fr(H,M).\end{equation}
\end{thm}

To prove this main result,  in Section 2  we introduce some auxiliary terms relevant to
our studies and give a crucial lemma that states that for a non-crossing compatible $M$-alternating set of $H$ with two members whose interiors have a containment relation,  the maximum anti-forcing number of $H$ is larger than the cardinality by at least two.
In Section 3, by using  this  lemma we obtain a stronger result: for any perfect matching $M$ of $H$ whose  anti-forcing number reaches the maximum value or minus one,  any two members in any given maximum non-crossing compatible $M$-alternating set of $H$ have disjoint  interiors and any member bounds a linear hexagonal chain, then give a proof of Theorem \ref{main-thm}.
 In Section 4 we give a complete characterization  to   hexagonal systems $H$ that always have   Eq. (1) for each perfect matching $M$ of $H$ by forbidding a triphenylene as nice subgraph.


\section{A crucial lemma}

In what follows, we  assume that all hexagonal systems are embedded in the plane such that
some edges parallel to each other are  vertical except for Figure \ref{fig:hexagonal-system-11}.
A \emph{peak} (resp. \emph{valley}) of a hexagonal system is a vertex whose neighbors
are below (resp. above) it.
For convenience, the vertices of a hexagonal system are colored with white and black such that
any two adjacent vertices receive different colors, and the \emph{peaks} are colored black.

Let $G$ be a plane bipartite graph with a perfect matching $M$, and
let $\mathcal{A}$ be a compatible $M$-alternating set of $G$.
Two cycles in $\mathcal{A}$ are said to be \emph{non-crossing} if
their interiors either are disjoint or have a containment relation.
Further, we say $\mathcal{A}$ is \emph{non-crossing} if any two cycles in $\mathcal{A}$ are non-crossing.
The following useful lemma was described in the first claim of the proof of {\cite[Theorem 3.1]{Lei-Zhang}}.

\begin{lem}\cite{Lei-Zhang}\label{Crossing-Noncrossing}
Let $H$ be a hexagonal system with a perfect matching $M$. Then
for any compatible $M$-alternating set $\mathcal{A}$ of $H$,
$H$ has a non-crossing compatible $M$-alternating set $\mathcal{A}'$
with $|\mathcal{A}'|=|\mathcal{A}|$.
\end{lem}

We now state a crucial lemma as follows.

\begin{lem}\label{Crucial-lemma}
For a hexagonal system $H$ with   a perfect matching $M_0$, let $\mathcal{A}_0$
be a non-crossing compatible $M_0$-alternating set of $H$.  Suppose
$\mathcal{A}_0$ has a pair of members so that their interiors have a containment relation.
Then $Af(H)\geq |\mathcal{A}_0|+2$.
\end{lem}

In order to prove this lemma, we need some further terminology.
Let $M$ be a perfect matching of $H$.
An edge of $H$ is called an $M$-\emph{double} \emph{edge} if
it belongs to $M$, and an $M$-\emph{single} \emph{edge} otherwise. $M$-double edges are often indicated by bold or double edges in figures.
An $M$-alternating cycle $C$ of $H$ is said to be \emph{proper} (resp. \emph{improper}) if
each $M$-double edge in $C$ goes from white end to black end (resp. from black end to white end)
along the clockwise direction of $C$.
The boundary of the infinite face of $H$ is called the \emph{boundary} of $H$,
denoted by $\partial(H)$. An edge on the boundary of $H$ is a \emph{boundary edge}.
A hexagon of $H$ is called an \emph{external hexagon} if
it contains a boundary edge, and an \emph{internal hexagon} otherwise.

A hexagonal system $H$ is \emph{cata-condensed} if all vertices of $H$ lie on $\partial(H)$.
A hexagon of a cata-condensed hexagonal system is a \emph{branch} if
it has three adjacent hexagons.
For example, the graph showed in Figure \ref{fig:counterexample-1}(b)
is a cata-condensed hexagonal system with exactly one branch. A cata-condensed hexagonal system without branch is a \emph{ hexagonal chain}. In particular, it is  a \emph{linear chain} if
the centers of all hexagons lie on a straight line. A maximal linear chain of a hexagonal chain is called a {\em segment}.

The symmetric difference of finite sets $A_1$ and $A_2$ is defined as
$A_1\oplus A_2:= (A_1\cup A_2)\backslash (A_1\cap A_2)$.
This operation can be defined among many finite sets in a natural way and
is associative and commutative.
If $C$ is an $M$-alternating cycle of $H$, then  $M\oplus C$ is also a perfect matching of $H$
and $C$ is an $(M\oplus C)$-alternating cycle of $H$, where $C$ may be regarded as its edge-set.

For a cycle $C$ of a hexagonal system $H$, let $I[C]$ denote the subgraph of $H$ formed by
$C$ together with its interior, and let $h(C)$ be the number of hexagons in $I[C]$.
\vskip 0.2cm

\noindent\textbf{Proof of Lemma \ref{Crucial-lemma}.}
By the assumption, we can choose a perfect matching $M$ of $H$ and a non-crossing compatible $M$-alternating set $\mathcal{A}$
 of $H$ so that  the following three conditions hold.\\
(i) $|\mathcal{A}|=|\mathcal{A}_0|$, \\
(ii) $\mathcal{A}$ has a pair of members so that their interiors have a containment relation, \\
(iii) $h(\mathcal{A}):=\sum_{C\in \mathcal{A}}h(C)$ is as small as possible subject to (i) and (ii).

Since $\mathcal{A}$ is non-crossing, we have that for any two cycles in $\mathcal{A}$
their interiors either are disjoint or one contains the other one.
Hence the cycles in $\mathcal{A}$ form a \emph{poset} according to
the containment relation of their interiors.
Since each $M$-alternating cycle has an $M$-alternating hexagon in its interior (cf. [32]),
we immediately have the following claim.

\vskip 0.2cm
\noindent\textbf{Claim 1.} Every minimal member of $\mathcal{A}$ is a hexagon.
\vskip 0.2cm

By the choice of $\mathcal{A}$, we can see that $\mathcal{A}$ has at least one non-hexagon member.
Let $C$ be a minimal non-hexagon member of $\mathcal{A}$.
Then $C$ is an $M$-alternating cycle whose interior contains only minimal members of $\mathcal{A}$.
By Claim 1 $C$ contains at least one hexagon as a  member of $\mathcal{A}$ in its interior.
Set $H':=I[C]$. So it follows that $H'$ is not a linear hexagonal chain.

\vskip 0.2cm
\noindent\textbf{Claim 2.} For any $M$-alternating hexagon $h$  in $H'$,
either  $h\in \mathcal{A}$ or at least one of the three $M$-double edges of $h$ does not belong to $C$.

\begin{proof}
If $h$ belongs to $\mathcal{A}$, then the claim holds.
If not, suppose to the contrary that the three $M$-double edges of $h$ belong to $C$.
Then $M\oplus h$ is a perfect matching of $H$,
and all (one to  three) components of $C\oplus h$ are $(M\oplus h)$-alternating cycles.
We can see that every minimal member of $\mathcal{A}$ in $H'$ is disjoint with $h$.
By the choice of $C$, $C\oplus h$ has a component as a cycle $C'$ which
is not a hexagon and contains a minimal member of $\mathcal{A}$ in its interior.
Since each vertex of $H$ has degree 2 or 3,
each $M$-double edge of $H$ is contained in at most two cycles of $\mathcal{A}$.
This implies that $\mathcal{A}\backslash \{C\}$ has at most one member intersecting $h$.
If such a member exists, denote it by $C''$ and
let $\mathcal{A'}:=(\mathcal{A}\cup \{h, C'\})-\{C, C''\}$;
otherwise, let $\mathcal{A'}:=(\mathcal{A}\cup \{C'\})-\{C\}$.
Then $\mathcal{A'}$ is a compatible $(M\oplus h)$-alternating set of $H$
satisfying Conditions (i) and (ii). But $h(\mathcal{A'})<h(\mathcal{A})$,
contradicting the choice for $\mathcal{A}$. Hence Claim 2 holds.
\end{proof}

 Now we focus our attention on  hexagonal system $H'$ with the boundary $C$ as some preliminaries.
Without loss of generality, suppose that $C$ is a proper $M$-alternating cycle (for the other case,
analogous arguments are implemented on right-top  and right-bottom corners of $H'$).

We apply an approach and notion appeared in Ref. \cite{Zhou-Zhang-2015}.  Along the boundary $C$ of $H'$,
we will find two substructures of $H'$ in its left-top corner and
left-bottom corner as Figures \ref{fig:hexagonal-system-01} and
 \ref{fig:hexagonal-system-03} respectively as follows.

A \emph{b-chain} of hexagonal system $H'$ is a maximal horizontal linear chain consisting of
the consecutive external hexagons when traversing (counter)clockwise the boundary  $\partial (H')$.
A  b-chain  is called  \emph{high} (resp. \emph{low})
if all hexagons adjacent to it are  below (resp. above) it.
For example,   in Figure \ref{fig:substructure} (taken from  \cite{Zhou-Zhang-2015}),
$D_0$, $D_1$, $D_2$, $G_1, G_2,\ldots, G_9, G_1', D_5, D_6$ and  $D_7$ are b-chains.
In particular, $D_0$, $D_1$, $D_2$ and $G_1$ are high b-chains,
while $G_1'$, $D_5$ and $D_6$ are low b-chains.
But $G_2, G_3, \ldots, G_9$ and  $D_7$ are neither high nor low b-chains.

\begin{figure}[H]
\begin{center}
\includegraphics[height=6.50cm]{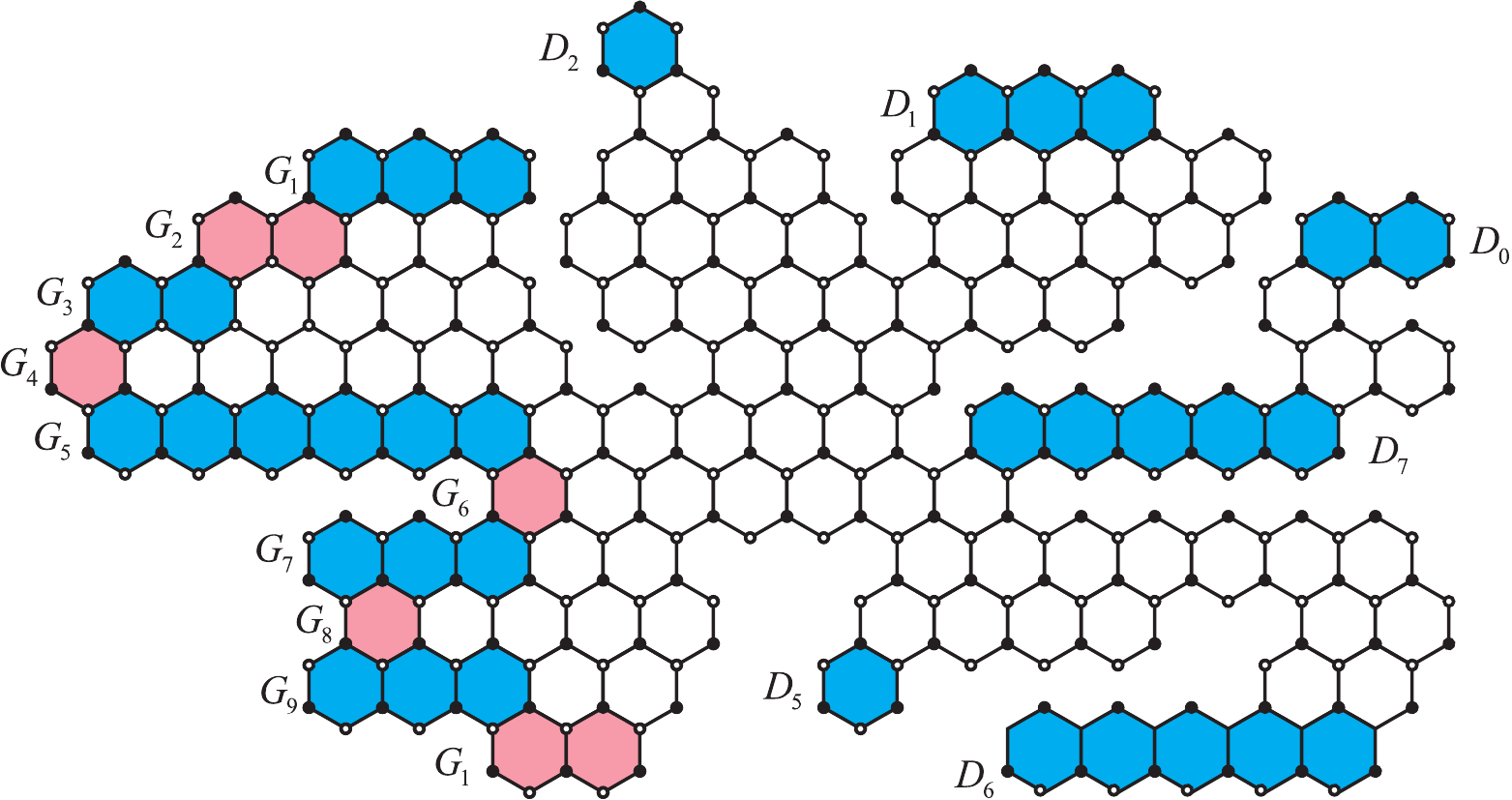}
\caption{{\small
Various b-chains of a hexagonal system, taken from \cite{Zhou-Zhang-2015}.}}
\label{fig:substructure}
\end{center}
\end{figure}

Given a high b-chain and a low b-chain of $H'$, they are distinct,
otherwise $H'$ itself is a linear chain,  contradicting the choice of $C$.
When traversing the b-chains along the boundary $\partial (H')$ counterclockwise from the high b-chain to the low b-chain,
let $G_1$ be the last high b-chain and let $G_1'$ be the first low b-chain after $G_1$.
Then the b-chains between $G_1$ and $G_1'$  descend monotonously.

From high b-chain $G_1$ we have a sequence of consecutive b-chains $G_1, G_2,\ldots, G_m$, $m\geq 1$,
with the following properties:
(1) for each $1\leq i<m$, $G_{i+1}$ is next to  $G_{i}$, and
the left end-hexagon of $G_{i+1}$ lies on the lower left side of $G_{i}$,
(2) either $G_m$  is just the low b-chain $G_1'$ or
$G_{m+1}$ is a b-chain next to $G_m$ such that  $G_{m+1}$
has no hexagon lying on the lower left side of $G_{m}$.
Let $G$ be the subgraph of $H'$ consisting of b-chains $G_1, G_2,\ldots, G_{m-1}$
and the hexagons of $G_m$ lying on the lower left side of $G_{m-1}$.
Then $G$ is a ladder-shape  hexagonal chain.

Similarly, from low b-chain $G_1'$ we have  a sequence of consecutive b-chains
$G_1', G_2',\ldots, G_s'$, $s\geq 1$, with the following properties:
(1) for each $1\leq j<s$, $G_{j+1}'$ is next to  $G_{j}'$, and
the left end hexagon of $G_{j+1}'$ lies on the upper left side of $G_{j}'$,
(2) either $G_s'$  is just the high b-chain $G_1$ or
$G_{s+1}'$ is next to the b-chain $G_{s}'$  such that  $G_{s+1}'$
 has no hexagon lying on the upper left side of $G_{s}'$.
Let $G'$ be the subgraph of $H'$ consisting of b-chains $G_1', G_2',\ldots, G_{s-1}'$
and the hexagons of $G_{s}'$ lying on the higher left side of $G_{s-1}'$. So $G'$ is an inverted  ladder-shaped hexagonal chain.

For example, given a high b-chain $D_1$ and a low b-chain $D_5$ in Figure \ref{fig:substructure},
we can get two required hexagonal chains $G=G_1\cup G_2\cup G_3\cup G_4$ and $G'=G_9\cup G_1'$.
The following claim is obvious.

\vskip 0.2cm
\noindent\textbf{Claim 3.} Either $G$ and $G'$ are disjoint or
they intersect only in the b-chain $G_m=G'_s$.
\vskip 0.2cm

To analyze the substructure $G$  of $H'$, we label some edges of $G$ as follows
(see Figure \ref{fig:hexagonal-system-01}):
let $e_{1,1}$ be the slant $M$-double edge of the right end hexagon of $G_1$ which
belongs to $C$ and contains a peak of  $H'$.
Neither $A$ nor $A'$ is contained in $H'$.
Denote by $e_{i,j}$, $1\leq i\leq m$ and  $1\leq j\leq n(i)$, the $j$-th edge of $G_i$ that  is parallel to $e_{1,1}$  and on the boundary  $C$ of $H'$,
and denote the specific   edges in $G_1$ and $G_m$
by $a, a'$ and $e_0, e_0'$ respectively.

\begin{figure}[H]
\begin{center}
\includegraphics[height=6cm]{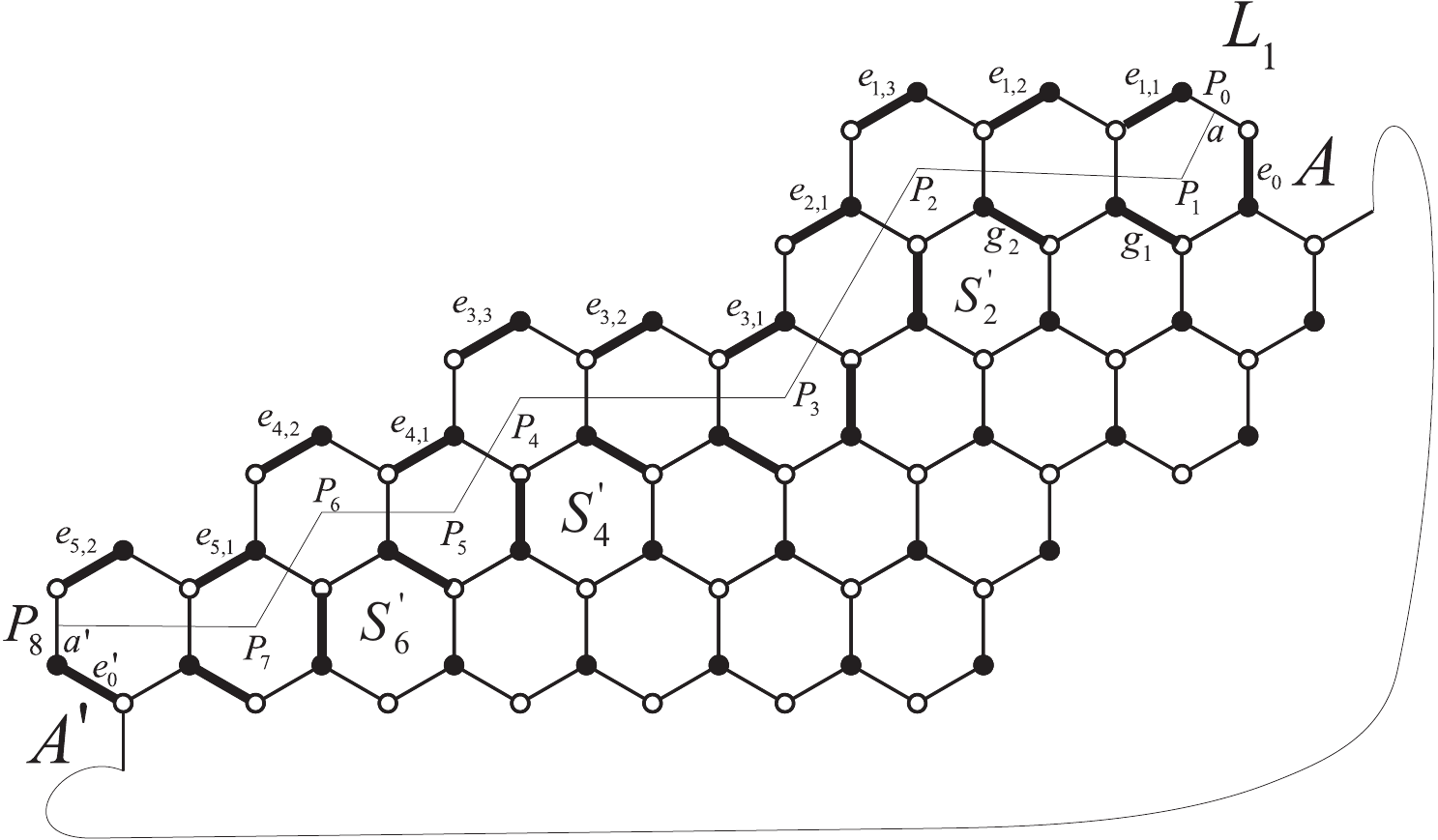}
\caption{{\small
Hexagonal chain $G$ on the left-top corner of $H'$
(bold edges are  $M$-double edges, $m=5$,
n(1)=3, n(2)=1, n(3)=3, n(4)=2, n(5)=2 and $A,A'\notin H'$) and the corresponding broken line segment $L_1$.}}
\label{fig:hexagonal-system-01}
\end{center}
\end{figure}

Since the boundary  $C$ of $H'$ is a proper $M$-alternating cycle,
all the edges $e_0$, $e_0'$, $e_{i,j}$, $1\leq i\leq m$, $1\leq j\leq n(i)$,
are $M$-double edges. In order to simplify our discussions,
we draw a ladder-shape broken line segment $L_1=P_0P_1\cdots P_{q+1}(q\geq 1)$
 (see Figure \ref{fig:hexagonal-system-01}) satisfying:
(1) $L_1$ only passes through hexagons of $G$,
(2) the endpoints $P_0$ and $P_{q+1}$  are the midpoints of the edges $a$ and $a'$ respectively,
(3) $L_1$ passes through the centers of all hexagons of $G$, and
(4) each $P_i$ ($1\leq i\leq q$) is a turning point, which is the center of a hexagon $S_i$ of $G$. Then
each line segment $P_iP_{i+1}$ ($0\leq i\leq q$) is orthogonal to an edge direction, and
 $P_{i+1}$ ($0\leq i\leq q$) lies on the lower left side or the left side of $P_i$
according as $i$ is even or odd.

Similarly we treat substructure $G'$ of $H'$ as follows  (see Figure \ref{fig:hexagonal-system-03}).
Let $f_{k,\ell}$, $1\leq k\leq s$ and $1\leq \ell\leq t(k)$, and $f_0,f_0', b,b'$
be a series of boundary edges on this structure as indicated in Figure
\ref{fig:hexagonal-system-03}.
Neither hexagon $B$ nor hexagon $B'$ is contained in $H'$.
Since the boundary of $H'$ is a proper $M$-alternating cycle,
we can see that all the edges $f_0$, $f_0'$, $f_{k,\ell}$,
$1\leq k\leq s$ and $1\leq \ell\leq t(k)$,
are $M$-double edges.
\begin{figure}[H]
\begin{center}
\includegraphics[height=4.5cm]{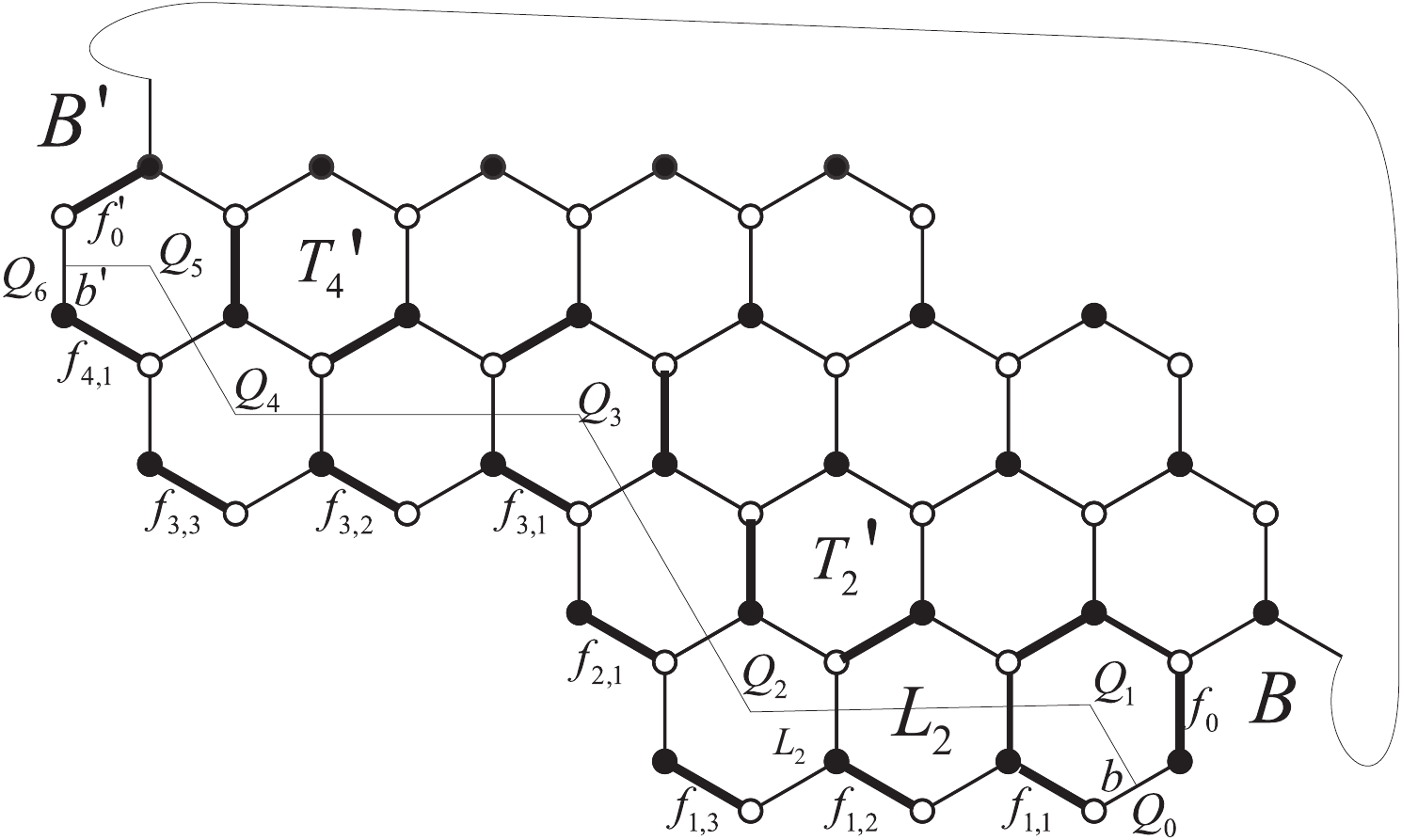}
\caption{{\small
Hexagonal chain $G'$ on the left-bottom corner of $H'$
(bold edges are  $M$-double edges, $s=4$, t(1)=3, t(2)=1, t(3)=3, t(4)=1
and $B,B'\notin H'$) and the corresponding broken line segment $L_2$.}}
\label{fig:hexagonal-system-03}
\end{center}
\end{figure}

Like $L_1$, we also draw a ladder-shape broken line segment
$L_2=Q_0Q_1\cdots Q_{r+1}(r\geq 1)$ as indicated in Figure
\ref{fig:hexagonal-system-03} so that  $L_2$ only passes through hexagons of $G'$ and each turning point $Q_i$ ($1\leq i\leq r$) is the center of a hexagon $T_i$ of $G'$.
It is obvious that both $L_i$, $i=1,2$, have an odd number of turning points.
By  Claim 3, we immediately obtain the following claim.

\vskip 0.2cm
\noindent\textbf{Claim 4.} Either the broken line segments $L_1$ and $L_2$ are disjoint
or the last segment $P_qP_{q+1}$ of $L_1$ is identical to the last segment $Q_rQ_{r+1}$ of $L_2$.
\vskip 0.2cm

Since the boundary of $H'$ is a proper $M$-alternating cycle,
we have that all the edges of $H$ intersected by $L_i$, $i=1,2$, are $M$-single edges.
We now have the following claim.

\vskip 0.2cm
\noindent\textbf{Claim 5.}
(a) The boundary of $G$ (resp. $G'$) is a proper $M$-alternating cycle, and\\
(b)  $n(1)=$ 1 or 2 (resp. $t(1)=$ 1 or 2), and $m\geq 2$ (resp. $s\geq 2$).

\begin{proof}
We only consider  $G$ (the other case is almost the same).
Let $Z_1$ be the path induced by those vertices of $G$ which are just upon $L_1$.
Let $Z_2$ be the path induced by those vertices of $G$ which are just below $L_1$.
Since the boundary $C$ of $H'$ is a proper $M$-alternating cycle,
$Z_1$ is an $M$-alternating path with two end edges in $M$.

To prove statement (a),  it suffices to show that
$Z_2$ is also an $M$-alternating path with two end edges in $M$.
Let $w_1(=e_0'), w_2, \ldots, w_{\ell_2}$ be all parallel edges  of $G$ below $P_qP_{q+1}$
and let $h_1(=e_0), h_2, \ldots, h_{\ell_1}$  be all vertical edges of $G$
on the right of $P_0P_1$   (see Figure \ref{fig:hexagonal-system-02}).
Note that  all the edges intersected by $L_1$ are $M$-single edges. It follows from $\{e_0, e_0'\}\subseteq M$ that
$h_1, h_2, \ldots, h_{\ell_1}$ (resp. $w_1, w_2, \ldots, w_{\ell_2}$)
are forced by $e_0$ (resp. $e_0'$) in turn to belong to $M$.

If $q=1$,  $Z_2$ is an $M$-alternating path with two end edges in $M$.
Let $q\geq 3$.  For each even $i$, $2\leq i\leq q-1$,
let $e_i''$ be the slant edge of $S_i$ in $Z_2$.
Let  $e_i$ and $e_i'$ be the two edges of $Z_2$ adjacent to $e_i''$
(see Figure \ref{fig:hexagonal-system-02}(a)).
Clearly, $e_i$ is parallel to $e_0$,  and $e_i'$ is parallel to  $e_0'$.
We assert that $e_i''\notin M$. Otherwise, $e_i''\in M$, and $e_i''$ does not lie on the boundary $C$ of $H'$
since $C$ is a proper $M$-alternating cycle.
So $H'$ has a hexagon $S_i'$ containing  $e_i, e_i'$ and $e_i''$.
Let $C':=C\oplus S_i$ and let $\mathcal{A}':=(\mathcal{A}\cup \{C'\})-\{C\}$ (see Figure \ref{fig:hexagonal-system-02}(b)).
Then $\mathcal{A}'$ is a compatible $M$-alternating set of $H$ satisfying conditions (i) and (ii).
But $h(\mathcal{A}')=h(\mathcal{A})-1$, contradicting the choice of $\mathcal{A}$.
So the assertion is true.
Note that  all the edges intersected by $L_1$ are $M$-single edges.
We have  $\{e_0, e_0', e_2, e_2', \ldots, e_{q-1}, e_{q-1}'\}\subseteq M$.
So it follows that $Z_2$ is an $M$-alternating path with two end edges in $M$
(see Figure \ref{fig:hexagonal-system-01}). Hence statement (a) holds.

\begin{figure}[H]
\begin{center}
\includegraphics[height=5cm]{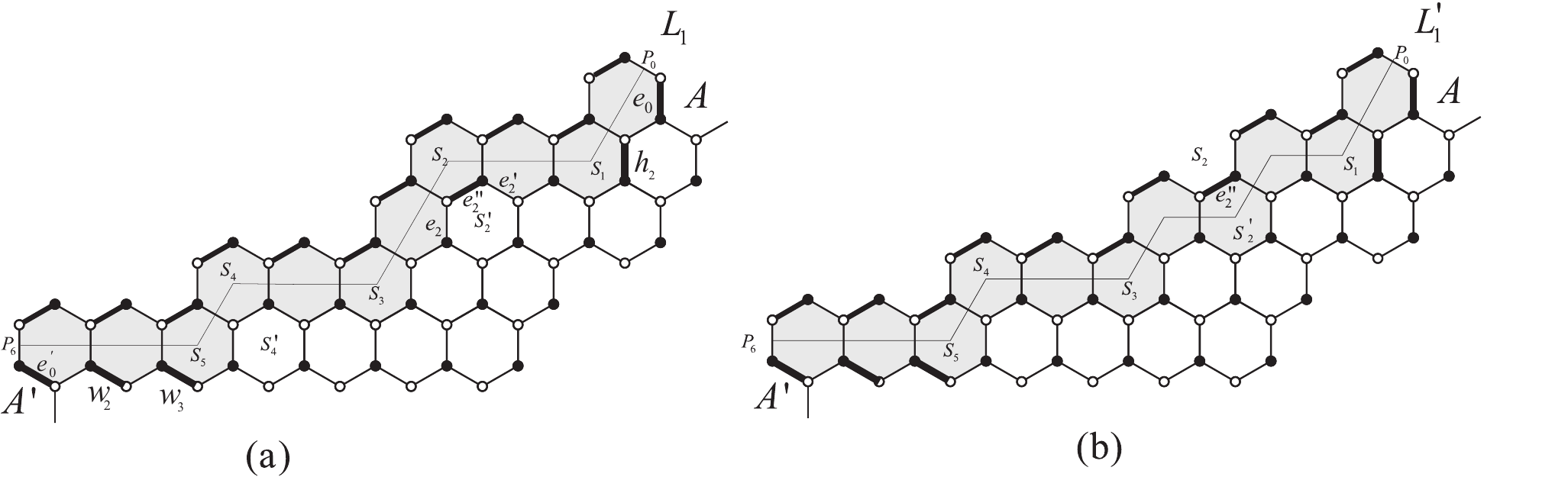}
\caption{{\small
Illustration for Claim 5 in the proof of Lemma \ref{Crucial-lemma}.}}
\label{fig:hexagonal-system-02}
\end{center}
\end{figure}
Next we prove statement (b).
Suppose to the contrary that $n(1)\geq 3$. Let $S_{1,1}$ and $S_{1,2}$ be the first and second hexagons of high b-chain $G_1$ from right to left. Then $P_1$ is the center of  $S_{1,1}=S_1$.
For $i=1,2$, let $g_i$ be the edge of $S_{1,i}$ parallel to $a$ and below $L_1$
(see Figure \ref{fig:hexagonal-system-01}).
By statement (a), we have  $g_1, g_2\in M$.
Therefore,    $S_1$  is a proper $M$-alternating hexagon, but not in $\mathcal{A}$.
By Claim 2, $g_1\notin C$. Since the boundary $C$ of $H'$ is an $M$-alternating cycle,
$g_1$ has no end-vertices in $C$. This implies that
$S_1$ has three consecutively adjacent hexagons in $H'$.
We can see that none of members of $\mathcal{A}$ except $C$ intersect $S_1$.
Let $M':=M\oplus S_1$ and  $\mathcal{A'}:=(\mathcal{A}\cup \{C\oplus S_1\})-\{C\}$.
Then  $M'$ is a perfect matching of $H$, and
$\mathcal{A'}$ is a compatible $M'$-alternating set satisfying conditions (i) and (ii).
But $h(\mathcal{A'})<h(\mathcal{A})$, contradicting the choice for $\mathcal{A}$.
Hence  $n(1)=$ 1 or 2.

Suppose to the contrary that $m=1$.
By statement (a), we can see that $g_1\in M$ and $g_1$ is an edge of $S_1$. We have that $S_1$  is a proper $M$-alternating hexagon, but not in $\mathcal{A}$.
By analogous  arguments as above, we arrive in
a similar contradiction no matter $n(1)=1$ or 2.
Hence $m\geq 2$ and  statement (b) holds.
\end{proof}

Claim 5 implies  that for all odd integers $i$ and $j$, $S_i$ ($1\leq i\leq q$) and  $T_j$ ($1\leq j\leq r$) are
 proper $M$-alternating hexagons, and the other hexagons of $G$ and $G'$ are not $M$-alternating.

For each even $i$, $2\leq i\leq q-1$, let $S_i'$ denote the hexagon (in the hexagonal lattice, but not necessarily contained in $H$)
adjacent to $S_i$ and below $L_1$ (see Figure \ref{fig:hexagonal-system-01}).
Similarly, for each even $j$, $2\leq j\leq r-1$, let $T_j'$ denote the hexagon
adjacent to $T_j$ and above $L_2$ (see Figure \ref{fig:hexagonal-system-03}).

By the above discussions to $H'$,  we now go back to the discussion to $H$ and will get our result. We now get a new perfect matching $M'$ of $H$ from $M$ by rotating all $M$-alternating hexagons of $G$ and $G'$ as  follows (see Figure \ref{fig:hexagonal-system-04}),
$$M':=M\oplus S_1\oplus S_3\oplus \cdots\oplus S_q\oplus T_1\oplus T_3\oplus\cdots\oplus T_r.$$

Let $\mathcal{B}$ be the set of $M'$-alternating hexagons in $G\cup G'$ and let $\mathcal{B}'=\{S_2',S_4',\dots,S_{q-1}'\}\\ \cup\{T_2',T_4',\dots,T_{r-1}'\}$. Then $\mathcal B\supseteq \{S_1, S_3, \dots, S_q, T_1, T_3,\dots,T_r\}$. We can have the following system of  cycles of $H$,
$$\mathcal{A}':=( \mathcal{A}\cup \mathcal{B})\setminus (\mathcal{B}'\cup \{C\}).$$

\begin{figure}[H]
\begin{center}
\includegraphics[height=6.5cm]{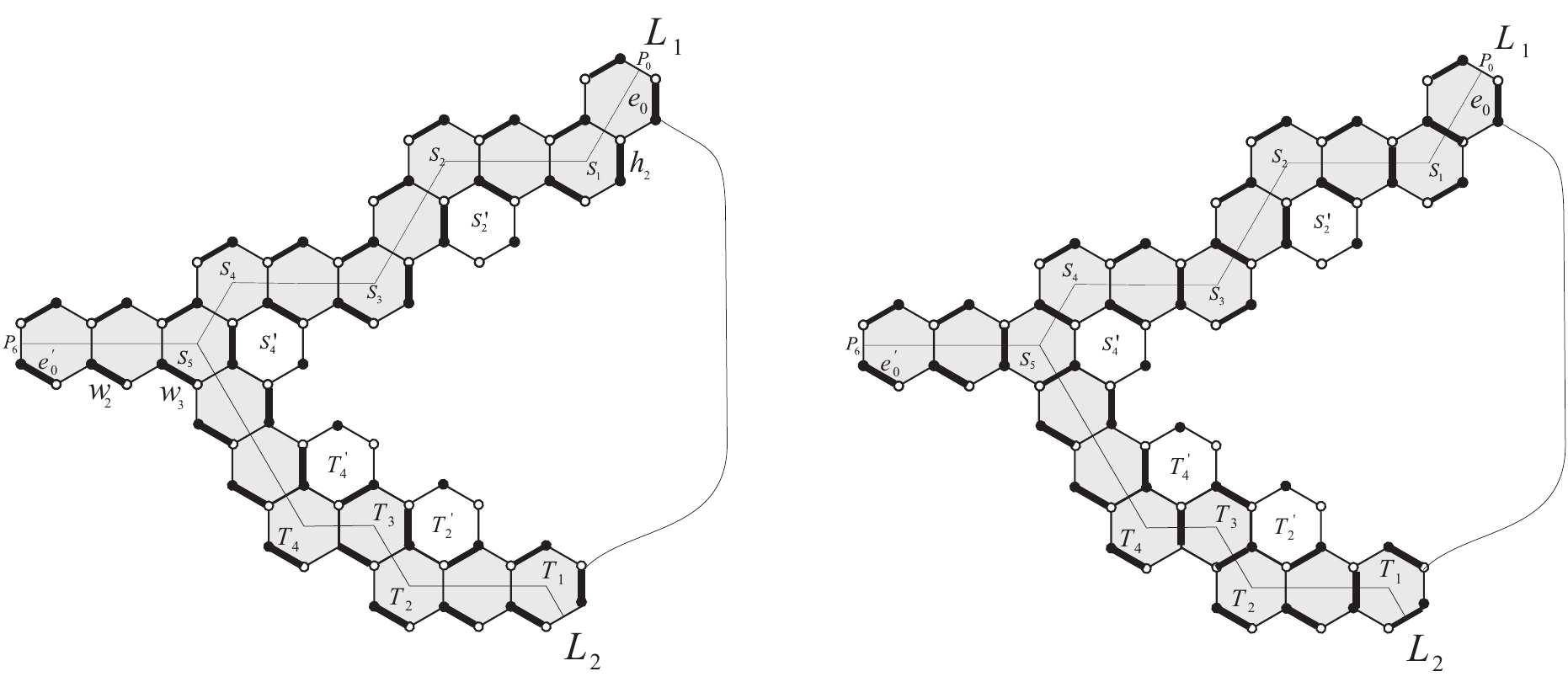}
\caption{{\small
Illustration for Claim 6: The gray hexagons form $G\cup G'$, and perfect matchings $M$ and $M'$ of $H$ have restrictions on $G\cup G'$ as left and right respectively.}}
\label{fig:hexagonal-system-04}
\end{center}
\end{figure}

\vskip 0.2cm
\noindent\textbf{Claim 6.} $\mathcal A'$ is an $M'$-alternating compatible set and $|\mathcal A'|\geq |\mathcal A|+2.$
\begin{proof}Given any member $C'$ in $\mathcal A'$. Then $C'\in \mathcal A$ or $\cal B$. First we want to show that $C'$ is an $M'$-alternating cycle.  If $C'$ does not intersect anyone of $S_1, S_3, \dots, S_q, T_1, T_3,\dots,T_r$, then $C'\in \mathcal A$ and $C'$ is both $M$- and $M'$-alternating cycle. If $C'=S_i$ or $T_j$ for odd $1\leq i\leq q$ and odd $1\leq j\leq r$, then $C'\in \mathcal B$ and $C'$ is both $M$- and $M'$-alternating cycle. The remaining case is that $C'$ intersects some $S_i$ or $T_j$ for odd $1\leq i\leq q$ and odd $1\leq j\leq r$, say the former $S_i$, but $C'\not=S_i$. We assert that $C'\in \mathcal B$, which implies that $C'$ is an $M'$-alternating hexagon. Suppose to the contrary that $C'\in \mathcal A\setminus \mathcal B$.  Then $C'$ is an $M$-alternating cycle. If $I[C']\subset H'=I[C]$, then $C'$ is an $M$-alternating hexagon not in $G$ since each member of $\mathcal A$ lying in the interior of $C$ is a hexagon. So $C'=S'_{i-1}$ for $i\geq 3$ or $C'=S'_{i+1}$ for $i\leq q-2$, a contradiction. Otherwise, $C'$ lies outside $C$ since  $C$ and $C'$ are non-crossing. Since  $C$ and $C'$ are compatible $M$-alternating cycles, $C'$ passes through only either the right vertical edge $e_0$ of $S_1$ for $n(1)\geq 2$ or $e_0'$ of $S_q$ for $n(m)=1$ and $G$ and $G'$ being disjoint. In such either case, three $M$-double edges of $S_1$ or $S_q$ belong to $C$, contradicting Claim 2, so the assertion holds.

Next we show that $\mathcal A'$ is an $M'$-alternating compatible set. For the members of $\mathcal A'$ lying in the interior of $C$, they are $M'$-alternating hexagons and thus compatible. For the members $C'$ of $\mathcal A'$ lying in the exterior of $C$, $C'$ is disjoint with everyone of $S_1, S_3, \dots, S_q, T_1, T_3,\dots,T_r$. Otherwise, $C'\in \mathcal B$ by the above assertion, contradicting that $C'$ lies on the exterior of $C$. So such members $C'$ are $M$-alternating cycles in $\mathcal A$ and compatible. Suppose that $C'$ intersects some member $h$ of $\mathcal A'$ inside $C$.
By the Jordan Curve Theorem we know that  $C'\cap h \subset C$. That is, each edge of $C'\cap h$ belong to $C$. Since $M$ and $M'$ have the same restriction on $C'\cap h$ and $C'$ and $C$ are compatible $M$-alternating cycles, each edge of $C'\cap h$ belong to  $M$, thus to $M'$, so $C'$ and $h$ are compatible $M'$-alternating cycles.

Finally we show the remaining inequality. For each odd $i$ with $1\leq i\leq q-2$, the hexagons between $S_i$ and $S_{i+2}$ in $G$ are not $M$-alternating, but at least one and at most two of them are  $M'$-alternating hexagons, which correspond to $S'_{i+1}$. Similarly for each odd $j$ with $1\leq i\leq r-2$, the hexagons between $T_i$ and $T_{i+2}$ in $G'$ are not $M$-alternating, but at least one and at most two of them are  $M'$-alternating hexagons, which correspond to $T'_{i+1}$. Next we consider the end segments of hexagonal chains $G$ and $G'$. If $n(1)\geq 2$, then $S_1$ is the right end-hexagon of $G_1$, and $S_1\notin \mathcal A$ since $S_1$ and $C$ are not compatible $M$-alternating cycles. But, $S_1\in \mathcal B$. Otherwise, $G_1$ is a single hexagon other than $S_1$, the upper end segment of $G$ has a unique $M$-hexagon and two $M'$-alternating hexagons, $S_1$ and its neighbor. Similarly we have that the last row $G_m$ of $G$ has more members of $\mathcal B$ than $\mathcal A$ by at least one. In analogous arguments as above we also have that the first segment and last row of $G'$ each has more members of $\mathcal A$ than $\mathcal B$ by at least one. Note that if both last rows of $G$ and $G'$ are identical, then their extra  members together count one, and $C$ is moved out $\mathcal A$. So we have that
$|\mathcal A'|\geq |\mathcal A|+2.$ \end{proof}

By Theorem 1.3 and Claim 6, we have
$Af(H)\geq af(H,M')=c'(H,M')\geq |\mathcal{A}'|\geq |\mathcal{A}|+2=|\mathcal{A}_0|+2$,
that is, $Af(H)\geq |\mathcal{A}_0|+2$.
Now the entire proof of the lemma is complete.
\hfill $\square$

%
%
%
%
%
%


\section{Minimax results for large anti-forcing numbers}

We can describe a minimax result stronger than Theorem \ref{main-thm} as follows.

\begin{thm}\label{containment-relation}
For a hexagonal system $H$, let $M$ be a perfect matching of $H$ with $af(H,M)=Af(H)$ or $Af(H)-1$,
and let $\mathcal{A}$ be a maximum non-crossing compatible $M$-alternating set of $H$.
Then (1) any two members in $\mathcal{A}$ have disjoint interiors, and
(2) for any  $C\in \mathcal{A}$, $I[C]$ is a linear chain.
\end{thm}

From Statement (1) of  Theorem \ref{containment-relation}, which is implied by Lemma \ref{Crucial-lemma},  we immediately obtain our main result. \\

\noindent\textbf{{Proof of Theorem} \ref{main-thm}.}
Let $\mathcal{A}$ be a maximum non-crossing compatible $M$-alternating set of $H$.
Then by Theorem \ref{Anti-forcing-feedback} and Lemma \ref{Crossing-Noncrossing},
we have $af(H,M)=|\mathcal{A}|$.
By Theorem \ref{containment-relation}(1), we know that for any two cycles in $\mathcal{A}$
their interiors are disjoint.
It was shown in \cite{Zhang-Zhang-2000} that for each $C\in \mathcal{A}$,
$H$ has an $M$-alternating hexagon $h$ in $I[C]$.
All such cycles $C$ in $\mathcal{A}$ are replaced with $M$-alternating hexagons $h$ of $I[C]$
to get a  set $K$ of $M$-alternating hexagons  with $|K|=|\mathcal{A}|$.
So we have $fr(H,M)\geq |K|=af(H,M)$.
On the other hand, $af(H,M)\geq fr(H,M)$ since $K$ is also a compatible $M$-alternating set. Both inequalities imply the result.
\hfill $\square$

\vskip 0.2cm

In order to prove Theorem \ref{containment-relation}(2),  the characterization for hexagonal systems $H$ with $af(H) = 1$ due to Li \cite{Li-1997} are presented here.
It is clear that $H$ has an anti-forcing edge if and only if $af(H) = 1$.

\begin{figure}[H]
\begin{center}
\includegraphics[height=3cm]{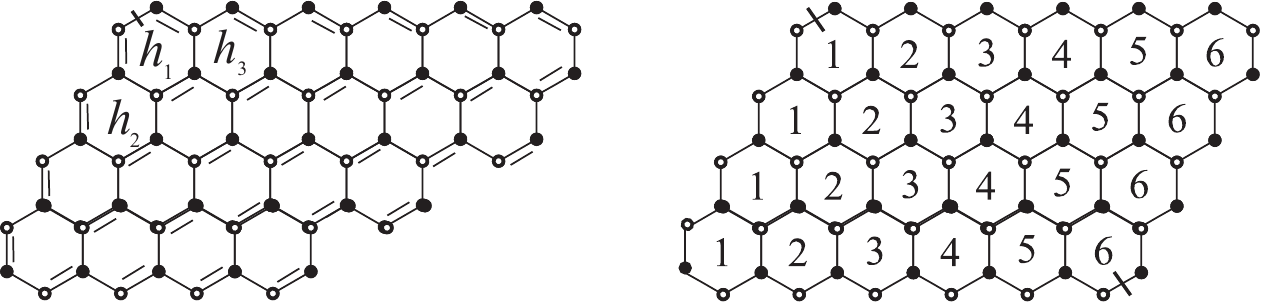}
\caption{{\small
Truncated parallelograms $H(6,6,5,4)$ and $H(6,6,6,6)$: anti-forcing edges are marked.}}
\label{fig:hexagonal-system-10}
\end{center}
\end{figure}

For integers $n_1 \geq n_2 \geq \ldots \geq n_k$, $k\geq 1$, let $H(n_1, n_2, \ldots, n_k)$
be a hexagonal system with $k$ horizontal rows of $n_1 \geq n_2 \geq \ldots \geq n_k$ hexagons and first
hexagon of each row being immediately below and to the right of the first one in the previous row, and
we call it \emph{truncated parallelogram} \cite{Cyvin-book-1988};
For example, see Figure \ref{fig:hexagonal-system-10}.
In particular, $H(r, r, \ldots, r)$ with $k \geq 2$ is parallelogram,
both $H(1, 1, \ldots, 1)$ with $k \geq 1$ and $H(r)$ with $r \geq 1$ are  linear chains.

\begin{thm}\label{anti-forcing-edge}
\cite{Li-1997} Let $H$ be a hexagonal system.  Then
 $af(H)=1$ if and only if $H$ is a truncated parallelogram.
\end{thm}


\vskip 0.2cm

\noindent\textbf{{Proof of Theorem} \ref{containment-relation}.}
(1) By Theorem \ref{Anti-forcing-feedback} and Lemma \ref{Crossing-Noncrossing}, we have $af(H,M)=|\mathcal{A}|$.
Suppose to the contrary that statement (1) does not hold. Then by Lemma \ref{Crucial-lemma}
we have $Af(H)\geq |\mathcal{A}|+2=af(H,M)+2\geq Af(H)+1$, a contradiction.
So statement (1) holds.

(2) Let $n:=af(H,M)=|\mathcal{A}|$ and let $\mathcal{A}=:\{C_1,C_2,\ldots,C_n\}$.
Choose an  anti-forcing set $S$ of $M$ with $|S|=n$.  Let $S_i:=S\cap E(C_i)$, $i=1,2,\ldots,n$.
By Lemma \ref{Anti-forcing-cycle}
we have $S_i\neq \emptyset$ for each $i$. Since $\mathcal{A}$ is a  compatible $M$-alternating set,
$S_i\cap S_j=\emptyset$ for any $1\leq i< j\leq n$. So we can assume that  $S:=\{e_1,e_2,\ldots, e_n\}$ with $e_i\in E(C_i)$ for all $1\leq i\leq n$.
For any $1\leq i \leq n$, since $C_i$ is an  $M$-alternating cycle of $H$,
the restriction $M_i$ of $M$ on $I[C_i]$ is a perfect matching of $I[C_i]$.
By Theorem \ref{containment-relation}(1), we can see that only  edge $e_i$ of $S$ lies in $I[C_i]$,
$1\leq i \leq n$.
So all $M_i$-alternating cycles in $I[C_i]$ pass through edge $e_i$, $1\leq i \leq n$.
By Lemma \ref{Anti-forcing-cycle}, we have that
$\{e_i\}$ is an anti-forcing set of $M_i$, $1\leq i \leq n$.
That is, $e_i$ is an anti-forcing  edge of $I[C_i]$. By Theorem \ref{anti-forcing-edge}, each $I[C_i]$ is a truncated parallelogram.

If some $I[C_i]$
 is not a linear chain, i.e. $I[C_i]$ has at least two rows and at least two columns of hexagons,
then $I[C_i]$ has a unique perfect matching $M_i$  not containing edge $e_i$
(see  Figure \ref{fig:hexagonal-system-10} (left)). Let $h_1$ be the hexagon of $I[C_i]$ with edge $e_i$, and let $h_2$ and $h_3$ be the two hexagons of $I[C]$ adjacent to $h_1$ in the first column and the first row respectively.
It was pointed out in \cite{Li-1997} that
$h_1$ is only $M_1$-alternating hexagon in $I[C_1]$.
So $M':=M\oplus h_1$ is a perfect matching of $H$, and $h_1,h_2$ and $h_3$ are $M'$-alternating hexagons.  We can see that  $\mathcal{A}':=\mathcal{A}\cup \{h_1, h_2, h_3\}-\{C\}$ is a compatible $M'$-alternating set of $H$.
Hence
$Af(H)\geq af(H,M')\geq |\mathcal{A}'|= |\mathcal{A}|+2\geq Af(H)+1$,
a contradiction. Hence each $I[C_i]$ is a linear chain and  statement (2) holds.
 \hfill$\Box$

\section{Minimax result for all perfect matchings}

It is natural to ask whether Theorem \ref{main-thm} holds for all perfect matchings $M$ of a hexagonal system $H$.  A counterexample can show  that the minimax relation does not necessarily hold for a perfect matching of a hexagonal system $H$
with the third maximum anti-forcing number $Af(H)-2$.

\begin{figure}[H]
\begin{center}
\includegraphics[height=2.5cm]{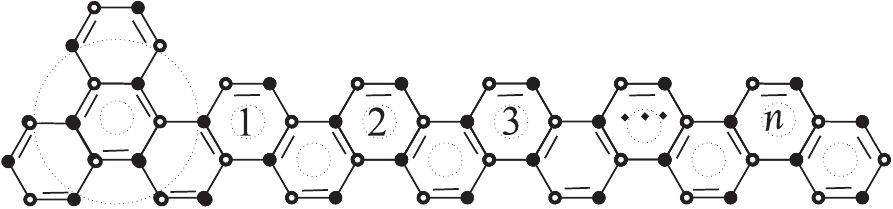}
\caption{{\small
A hexagonal system $R_n$ with a perfect matching $M$.}}
\label{fig:hexagonal-system-11}
\end{center}
\end{figure}

Let $R_n$ be a hexagonal system with $2n+4$ hexagons  and a perfect matching $M$ as shown in Figure
\ref{fig:hexagonal-system-11} (the edges in $M$ are indicated by double edges). Then $R_n$ contains one triphenylene whose central hexagon is denoted $h$. Let $M'=M\oplus h$. Then all hexagons of $R_n$ are $M'$-alternating. So
the Fries number of $R_n$ is the number of hexagons in $R_n$  (see also \cite{Harary-1991}).
By Theorem \ref{HS-F(H)-Fries(H)}, we have $Af(R_n)=Af(R_n,M')=Fr(R_n)=2n+4$.

However we can confirm that $af(R_n, M)=Af(H)-2>fr(H,M)$. By counting $M$-alternating hexagons in $R_n$, we have that  $fr(H,M)=2n+1$. On the other hand, we can find a compatible $M$-alternating set of size $2n+2$. So $af(R_n, M)\geq 2n+2> fr(H,M)$.
By a direct check or Theorem \ref{main-thm}, we have $af(R_n, M)\leq Af(R_n)-2$.
So $af(R_n, M)=Af(H)-2$.


%
%

We can see that the above counterexample contains a triphenylene as nice subgraph. In fact we can give a characterization for a hexagonal system $H$ to have the mini-max relation $af(H,M)=fr(H,M)$ by forbidding triphenylene as nice subgraph (see Theorem \ref{Triphenylene}).

To the end we present some concepts and  known results. Let $H$ be a hexagonal system with a perfect matching.  Let $r(H)$ and $k(H)$ be the numbers of sextet patterns and Kekul\'e structures of $H$ respectively.

\begin{thm}[\hskip -0.02mm\cite{Zhang-Chen-1986, Shiu-2002}]\label{Coronoid-nice-subgraph}
Let $H$ be a hexagonal system with a perfect matching. Then $r(H)\leq k(H)$, and the following statements are equivalent.\\
(i) $r(H)= k(H)$,\\
(ii) $H$ has a coronene as a nice subgraph, and\\
(iii)  $H$ has two disjoint cycles $R$ and $C$ so that  $R$ lies in the interior of $C$ and  $R\cup C$ is a nice subgraph of  $G$.
\end{thm}

\begin{figure}[H]
\begin{center}
\includegraphics[height=5.5cm]{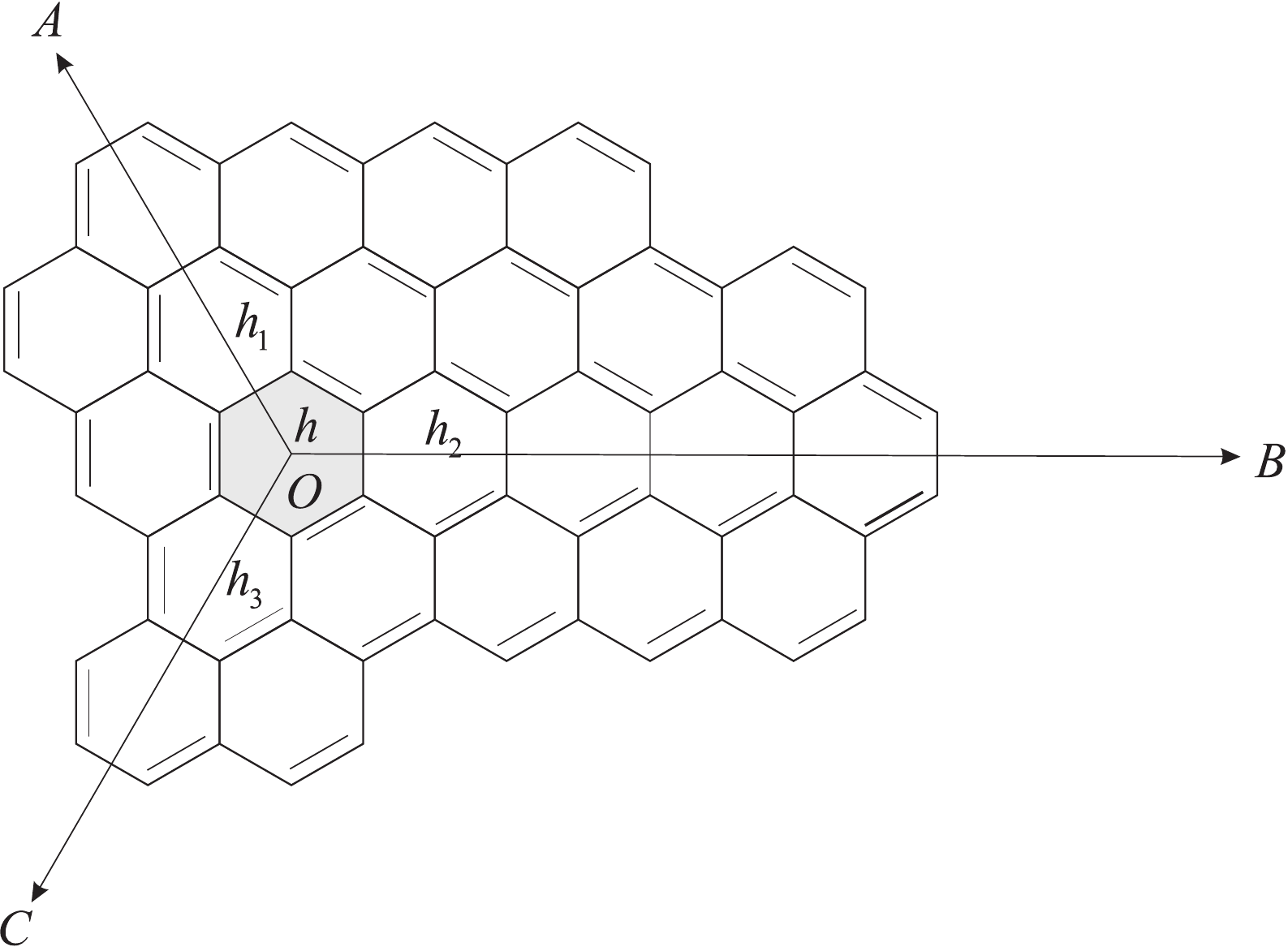}
\caption{A hexagonal system with a unique $M$-alternating hexagon for  a perfect matching $M$  marked by double lines.}
\label{Chapter-5-12}
\end{center}
\end{figure}

Let $H$ be a hexagonal system with a hexagon $h$. Draw three rays $OA, OB$ and $OC$ from the center $O$ of $h$ so that they pass through the centers of three disjoint edges of $h$ respectively, which divide the plane  into three areas $AOB, BOC$ and $COA$. Such three regional coordinate system is denoted by $O-ABC$.　  Zhang et al. \cite{Zhang-Guo-Chen-1988-2} ever gave the following fact.

\begin{lem}\label{Chapter-fact}
A hexagonal system $H$ has  a perfect matching $M$ with a unique $M$-alternating hexagon $h$  if and only if it has the coordinate system $O-ABC$ so that $O$　 is the center  of $h$,　rays $OA, OB$ and $OC$ do not intersect edges in $M$ and all edges of $M$ in anyone of areas  $AOB, BOC$ and  $COA$ are parallel to each other.
\end{lem}

%
%
%
%

\begin{thm}\label{Triphenylene}
Let $H$ be a hexagonal system with a perfect matching. Then $H$ has no triphenylenes as nice subgraph if and only if for each perfect matching $M$ of $H$, $af(H,M)=fr(H,M)$.
\end{thm}

\begin{proof}
We first prove the sufficiency. Suppose to the contrary  that  $H $ contains a triphenylene as a nice subgraph. Let $M'$ be a perfect matching of the triphenylene as  shown in Figure \ref{fig:counterexample-1}(b). Because the triphenylene is a nice subgraph of  $H$,  $M'$ can be extended to a perfect matching  $M $ of  $H $. So  $M'\subseteq M $. Let  $C$ be  the  boundary of the triphenylene. It is easy to see that $C $ is an $M$-alternating cycle of $H$ and $C$  is compatible with each  $M$-alternating hexagon of $H$. By Theorem \ref{Anti-forcing-feedback}, we have $af(H,M)\geq fr(H,M)+1>fr(H,M)$, a contradiction. Hence  the sufficiency holds.

We now prove the necessity. Suppose to the contrary  that $H $ has a perfect matching $M_0 $ so that  $af(H, M_0) > fries (H, M_0)$. Let $\mathcal {A} $ be a maximum non-crossing compatible  $M_0 $-alternating  set of  $H$. By Theorem \ref{Anti-forcing-feedback} and Lemma \ref{Crossing-Noncrossing}  we have $af(H, M_0) = |\mathcal {A}|$ and there are two cycles in   $\mathcal {A}$ so that their interiors have a containment relation; Otherwise, since  there is an  $M_0$-alternating hexagon  in the interior of each  $M_0$-alternating cycle  \cite{Zhang-Zhang-2000},  $H$ has at least  $af(H, M_0)$ $M _ 0 $-alternating hexagons, that is, $af(H, M_0)\leq fr(H, M_0)$, a contradiction. So we can select two cycles  $C_1$ and $C_2$ in $H$ to meet the following conditions:\\
\indent (i) $I[C_1] \subseteq I[C_2] $, \\
 \indent(ii) $H$ has a perfect matching $M$ so that  $C_1$ and  $C_2$ are  compatible  $M $-alternating cycles, and\\
 \indent (iii) $h(C_1) + h(C_2)$ is as small as possible subject to Conditions  (i) and (ii) (recall that   $h(C_i) $ is  the number of hexagons inside  $C_i$).

If  $I[C_2]$ has an $M$-alternating  hexagon  $h$ which is disjoint with  $C_2$, then
    $h\cup C_2$ is a nice subgraph of  $H$. By Theorem \ref{Coronoid-nice-subgraph}, $H$ contains a coronene  as a nice subgraph. So $H$ also contains triphenylene as a nice subgraph.

    From now on suppose that all the  $M $-alternating hexagons in $I[C_2]$ intersect $C_2$.
    Obviously, if  $C_2$ is a proper (resp. improper)  $M$-alternating cycle, then  each of the  $M$-alternating hexagons in  $I[C_2]$ is also proper (resp. improper). Without loss of generality, suppose that  the  $C_2$ and   $M$-alternating hexagons in $I[C_2]$ are proper. Take an $M $-alternating hexagon  $h$  inside  $C_1 $. Since $C_1$ and $C_2$ are compatible  $M $-alternating,   $h$  is compatible with  $C_2$. We can show the following fact.
\vskip 2mm
    \noindent\textbf {Claim.}  $h$ is the only  $M$-alternating hexagon in $I[C_2]$.
    \begin{proof}Suppose to the contrary  that  $h'$ is an $M$-alternating hexagon of $I[C_2]$ different from  $h$. Then  $h'$ and  $h$ are disjoint  because any two proper   $M $-alternating hexagons do not intersect. Let $M': = M\oplus h'$ and $C'_2: =C_2\oplus h'$. Then $M'$ is a perfect matching of  $H$, and  each component of  $C_2\oplus h'$ is  an $M'$-alternating cycle. Take a component $C'_2$ of $C_2\oplus h'$ so that $h$ lies inside   $C'_2$. Then  $h$ and  $C_2'$ are  also compatible  $M'$-alternating cycles.  It is clear that the cycles  $C'_2 $ and  $h$ satisfy the above conditions  (i) and  (ii). But  $h(C'_2)+1<h(C_2)+h(C_1)$,  a  contradiction with the minimality of $h(C_2)+h(C_1)$.\end{proof}

     Note that the restriction of $M$ on $I[C_2]$ is a perfect matching of  $I[C_2]$. From the Claim  and Lemma \ref{Chapter-fact}, from center $O $ of  $h$ we establish coordinate system $O-ABC$ so that  $OA, OB$ and $OC$ do not pass through $M$-edges in  $I[C_2]$, and all $M$-double edges of $I[C_2]$ in anyone of areas $AOB, BOC$ and $COA$ are parallel to each other (see Figure \ref {Chapter-5-12}). Because  $h$ and  $C_2$ are two compatible  $M $-alternating cycles,  every  $M $-single edge of  $h $ is not on  $C_2$. This shows that  $h$ has three adjacent hexagons  $h_1, h_2$, and $h_3$ in $I[C_2]$ which intersect $OA, OB$, and $OC $  respectively. Further,  hexagons  $h, h_1, h_2 $ and  $h_3$ form a triphenylene whose boundary is $M$-alternating cycle. So the triphenylene is a nice subgraph of  $H$, a contradiction. So $af(H, M_0) =fr(H, M_0)$ and the necessity holds.
\end{proof}


\end{document}